\documentclass[11pt]{amsart}
\usepackage{amssymb,amsmath,mathrsfs}
\usepackage{enumerate,enumitem,tikz}
\usepackage{hyperref,mathtools}

\usepackage[vcentermath]{youngtab}
\usepackage{xstring}
\usepackage{array,booktabs}

\usetikzlibrary{patterns}
\usetikzlibrary{decorations.pathreplacing}

\newtheorem{thm}{Theorem}[section]
\newtheorem{prop}[thm]{Proposition}

\newtheorem{lem}[thm]{Lemma}

\theoremstyle{definition}
\newtheorem{example}[thm]{Example}
\newtheorem*{remark}{Remark}

\numberwithin{equation}{section}

\newcommand{\Mo}{\mathfrak{M}}
\newcommand{\MoP}{\mathcal{M}}
\newcommand{\Rio}{\mathfrak{R}}
\newcommand{\RioP}{\mathcal{R}}
\newcommand{\Uu}{\mathsf{U}}
\newcommand{\Dd}{\mathsf{D}}
\newcommand{\Hh}{\mathsf{H}}
\newcommand{\LR}{\operatorname{LR}}
\newcommand{\SYT}{\operatorname{SYT}}
\newcommand{\RSYT}{\SYT^{\,\equiv}_{\le 3}}

\DeclarePairedDelimiter{\abs}{\lvert}{\rvert}

\DeclareMathOperator{\lrm}{\mathsf{lrmax}}

\def\oeis#1{\cite[#1]{Sloane}}

\newcommand\motzkinpath[3]{%
  \begin{scope}
    \def\diam{0.08}%
    \draw[help lines] (#1) -- ++({(#2)},0);
    \draw[line width=1pt] (#1) foreach \s in {#3} {-- ++(1,{\s-1})};
    \filldraw (#1) circle (\diam);
    \filldraw (#1) foreach \s in {#3} {++(1,{\s-1}) circle (\diam)};
  \end{scope}
}

\newcommand{\perm}[1]{%
  \def\Perm{#1}
  \StrSubstitute{\Perm}{,}{\,}
}

\newcommand{\plotpermutation}[1]{%
\def\Permutation{#1}

\StrDel{\Permutation}{,}[\OnelineP]
\StrLen{\OnelineP}[\N]

\def\rad{0.15}
\begin{scope}
\draw[gray!60] (1,1) grid (\N,\N);
\foreach [count=\i] \y in \Permutation {%
  \draw[fill,right=14pt,above=14pt] (\i,\y) circle (\rad);
  }
\end{scope}
}

\begin{document}
\title[Motzkin paths, 321-avoiding permutations, SYT with rows of equal parity]{Motzkin paths, 321-avoiding permutations, and standard Young tableaux with rows of equal parity}
\author[Ryan Amaral]{Ryan J. Amaral}
\address{Penn State Altoona\\ 3000 Ivyside Park\\ Altoona, PA 16601}
\email{rva5595@psu.edu}
\author[Juan Gil]{Juan B. Gil}
\address{Penn State Altoona\\ 3000 Ivyside Park\\ Altoona, PA 16601}
\email{jgil@psu.edu}
\author[Michael Weiner]{Michael D. Weiner}
\address{Penn State Altoona\\ 3000 Ivyside Park\\ Altoona, PA 16601}
\email{mdw8@psu.edu}

\begin{abstract}
Motzkin paths of length $n$ and standard Young tableaux (SYT) with $n$ cells and at most three rows are both counted by the Motzkin numbers, and many bijections between them are known. The Riordan numbers count the subfamilies of Riordan paths (Motzkin paths with no horizontal step on the $x$-axis) and of tableaux whose three row lengths have the same parity, but none of the known bijections restricts to these subfamilies. We introduce the set of $321$-avoiding permutations in which every left-to-right maximum is either a descent or a fixed point. This family is counted by the Motzkin numbers, and its fixed-point-free elements are the ``Riordan permutations'' of Callan. We give a bijection from Motzkin paths to these permutations under which Riordan paths correspond to Riordan permutations. We then give a bijection from these permutations to SYT of height at most three, obtained from Robinson--Schensted insertion followed by a parity correction, under which Riordan permutations correspond to tableaux with rows of equal parity and the number of left-to-right maxima becomes a simple tableau statistic. Via Dyck paths, we connect these objects to further families counted by the Riordan numbers, including derangements of genus zero and SYT of shape $(k,k,1^{n-2k})$.
\end{abstract}

\maketitle

\section{Introduction}

A {\em Motzkin path} of length $n$ is a lattice path from $(0,0)$ to $(n,0)$ with steps $\Uu=(1,1)$, $\Hh=(1,0)$, and $\Dd=(1,-1)$ that never goes below the $x$-axis. A {\em Riordan path} is a Motzkin path with no $\Hh$-step on the $x$-axis. It is known that Motzkin paths are counted by the sequence \oeis{A001006}:
\[ 1, 1, 2, 4, 9, 21, 51, 127, 323, 835, 2188, 5798, 15511, 41835, 113634, \dots, \]
and Riordan paths by the sequence \oeis{A005043}:
\[ 1, 0, 1, 1, 3, 6, 15, 36, 91, 232, 603, 1585, 4213, 11298, 30537,\dots, \]
both starting at $n=0$. Since the only Motzkin path of length $1$ is a single $\Hh$-step, which lies on the $x$-axis, there is no Riordan path of length $1$; the $3$ Riordan paths of length $4$ are $\Uu\Dd\Uu\Dd$, $\Uu\Uu\Dd\Dd$, and $\Uu\Hh\Hh\Dd$.

Motzkin numbers also count the standard Young tableaux ($\SYT$) with $n$ cells and at most three rows, and several bijections realizing this are known; see for instance Matsakis and Vandervelde \cite{MV22}. We denote this set by $\SYT_{\le3}(n)$. The Riordan numbers count a natural subfamily of these tableaux. Consider the partition set
\[ \Lambda^{\equiv}_{\le3}(n)
  = \bigl\{ \lambda = (\lambda_1,\lambda_2,\lambda_3) :
    \lambda_1 \ge \lambda_2 \ge \lambda_3 \ge 0,\ \textstyle\sum\limits_i \lambda_i = n,\
    \lambda_1 \equiv \lambda_2 \equiv \lambda_3 \!\! \pmod 2 \bigr\}, \]
and let
\[  \RSYT(n) = \bigsqcup_{\lambda \in \Lambda^{\equiv}_{\le3}(n)} \SYT(\lambda). \]
In other words, $\RSYT(n)$ is the set of standard Young tableaux with $n$ cells and at most three rows whose three row lengths all have the same parity, where empty rows are included and counted as even.

\begin{thm}\label{thm:known}
For every $n\ge 0$, we have $\abs{\RSYT(n)}=R_n$, the $n$-th Riordan number.
\end{thm}

This result is due to Okada \cite[Cor.~5.5(3)]{Oka16}, obtained from Pieri rules for the orthogonal groups. Jagenteufel \cite{Jag18} gave a bijective proof; in fact, as a byproduct of a correspondence for $\mathrm{SO}(3)$, she constructed an explicit bijection from Riordan paths of length $n$ to $\RSYT(n)$. Motivated by a question of Amdeberhan, the result was recently rediscovered and proved independently by Hemmer, Straub and Westrem~\cite{HSW26}.

Although several bijections between Motzkin paths of length $n$ and $\SYT_{\le3}(n)$ are available, restricting any of them to Riordan paths yields tableaux whose shapes are not all in $\Lambda^{\equiv}_{\le3}(n)$. The authors of \cite{HSW26,HSW26b} therefore asked for a bijection at the Motzkin level that naturally restricts to a bijection between the corresponding Riordan objects. Jagenteufel's map gives a bijection at the Riordan level, but it is not known to extend to a global bijection at the Motzkin level. In this paper we address the question via pattern-avoiding permutations.

Specifically, we consider the set $\Mo_n$ of $321$-avoiding permutations $\pi$ of size $n$ for which each left-to-right (LR) maximum $\pi_i$ satisfies $\pi_i > \pi_{i+1}$ or $\pi_i = i$; for $i=n$ only the latter can occur. Moreover, we let $\Rio_n$ be the set of fixed-point-free permutations in $\Mo_n$. It is known that $\abs{\Rio_n}=R_n$, and we will prove that $\abs{\Mo_n}=M_n$, the $n$-th Motzkin number. The set $\Rio_n$ coincides with the set of permutations listed in \oeis{A005043} by D.~Callan, namely the $321$-avoiding permutations in which every LR maximum is a descent. Indeed, a fixed point $j$ of a $321$-avoiding permutation $\pi$ is never a descent. We also point out that our set $\Mo_n$ is different from the set of Motzkin permutations studied by Elizalde and Mansour~\cite{EM05}.

Our main goal is to construct an explicit bijection $\Phi$ from $\Rio_n$ to $\RSYT(n)$ that extends to a bijection $\widetilde\Phi$ from $\Mo_n$ to $\SYT_{\le3}(n)$. The map $\Phi$ is obtained by applying Robinson--Schensted insertion to a permutation $\theta(\pi)$ built from the $\LR$ maxima of $\pi$, and then correcting the parity of the resulting shape by at most two elementary moves. Together with a natural bijection $\alpha$ between $\Mo_n$ and Motzkin paths of length $n$, which carries $\Rio_n$ onto the Riordan paths (Theorem~\ref{thm:path2perm_map}), this answers the question of \cite{HSW26,HSW26b}: the composition $\widetilde\Phi\circ\alpha^{-1}$ is a bijection from Motzkin paths of length $n$ to $\SYT_{\le3}(n)$ whose restriction to Riordan paths is a bijection onto $\RSYT(n)$ (Theorem~\ref{thm:path_syt}). Moreover, $\Phi$ carries the number of $\LR$ maxima of $\pi$, which corresponds to the number of $\Dd$-steps of $\alpha(\pi)$, to a simple statistic on $\RSYT(n)$. Composing further with standard bijections involving Dyck paths, we obtain explicit bijections to other families counted by the Riordan numbers, including derangements of genus zero and $\SYT$ of shape $(k,k,1^{n-2k})$.

The paper is organized as follows. In Section~\ref{sec:path2perm_bijections} we define the map $\alpha$, prove that it is a bijection between $\Mo_n$ and Motzkin paths that restricts to $\Rio_n$ and Riordan paths, and describe a bijection $\hat\gamma$ from $\Mo_n\setminus\Rio_n$ to $\Rio_{n+1}$ that realizes the identity $M_n=R_n+R_{n+1}$ at the level of permutations. In Section~\ref{sec:perm2SYT_bijections} we construct the map $\Phi$, prove that it is a bijection from $\Rio_n$ to $\RSYT(n)$ by giving an explicit inverse, and use $\hat\gamma$ to extend it to a bijection $\widetilde\Phi$ from $\Mo_n$ to $\SYT_{\le3}(n)$. In Section~\ref{sec:dyck_paths}, we combine these results into the desired bijection from Motzkin paths to tableaux, and we connect the Riordan families studied here, via Dyck paths, to other Riordan families in the literature (Theorem~\ref{thm:riordan_families}), illustrating all of the bijections on a single example.

\section{Bijections between paths and permutations}
\label{sec:path2perm_bijections}

Let $\MoP_n$ and $\RioP_n$ denote the sets of Motzkin and Riordan paths of length $n$, respectively, and let $\Mo_n$ be the set of $321$-avoiding permutations $\pi$ of $[n]$ for which each LR maximum is either a descent or a fixed point. We will use two standard facts about a $321$-avoiding permutation $\pi$: its entries that are not LR maxima form an increasing subsequence, and if $\pi_j=j$, then $\pi=\sigma\oplus 1\oplus\tau$ with $\sigma\in\mathcal{S}_{j-1}(321)$ and $\tau\in\mathcal{S}_{n-j}(321)$. In particular, a fixed point of $\pi$ is necessarily a LR maximum. Note also that $\Mo_n$ is closed under $\oplus$, since descents and fixed points of each component are preserved.

We start by defining a map $\alpha:\Mo_n\to\MoP_n$. Let $\pi\in\Mo_n$, let $\LR(\pi)=\{\pi_{i_1},\dots,\pi_{i_k}\}$ be the set of its LR maxima, and let $\lrm(\pi) = \abs{\LR(\pi)}$. In addition, let
\[ D(\pi)=\{\pi_i : \pi_i\in\LR(\pi),\, \pi_i\ne i\} \;\text{ and }\; U(\pi)=\{\pi_{i+1} : \pi_i\in D(\pi)\}. \]
If $\pi_i\in D(\pi)$, then $\pi_i>\pi_{i+1}$, so $\pi_{i+1}$ is not a LR maximum. Hence $U(\pi)\cap D(\pi)=\varnothing$ and $\abs{U(\pi)}=\abs{D(\pi)}$. We let $w=\alpha(\pi)$ be the word defined by
\[ w_v=\begin{cases}
     \Dd &\text{if } v\in D(\pi),\\
     \Uu &\text{if } v\in U(\pi),\\
     \Hh &\text{otherwise.}
   \end{cases} \]
Since each $\Dd$ at index $\pi_i$ is matched with a $\Uu$ at the smaller index $\pi_{i+1}$, every prefix of $w$ contains at least as many $\Uu$'s as $\Dd$'s, so $w$ is a Motzkin path. Moreover, if $\pi=\sigma\oplus\tau$ with $\abs{\sigma}=m$, then position $m$ is not a descent of $\pi$ (as $\pi_{m+1}>m\ge\pi_m$), so no pair $\{\pi_i,\pi_{i+1}\}$ with $\pi_i\in D(\pi)$ splits into two different components. It follows that if $\pi=\sigma_1\oplus\cdots\oplus\sigma_r$, then $\alpha(\pi)$ is the concatenation $\alpha(\pi) = \alpha(\sigma_1)\cdots\alpha(\sigma_r)$.

For example, the permutation $\pi=\perm{2,1,6,3,4,7,5,8,10,9}$ in $\Mo_{10}$ corresponds to the Motzkin path $\alpha(\pi) = \Uu\Dd\Uu\Hh\Uu\Dd\Dd\Hh\Uu\Dd$, shown in Figure~\ref{fig:motzkinP}. The path can be read (bottom to top) from the labels on the left of the plot of $\pi$.

\begin{figure}[ht]
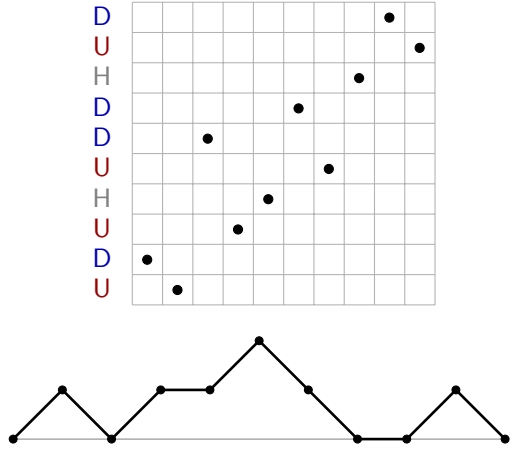

\tikz[scale=0.4]{
\plotpermutation{2,1,6,3,4,7,5,8,10,9}
\foreach \y in {2,6,7,10} {\node[above=-1pt] at (0,\y) {\small\color{blue!70!black} $\Dd$};}
\foreach \y in {1,3,5,9} {\node[above=-1pt] at (0,\y) {\small\color{red!60!black} $\Uu$};}
\foreach \y in {4,8} {\node[above=-1pt] at (0,\y) {\small\color{gray} $\Hh$};}
} \\[10pt]
\tikz[scale=0.65]{\motzkinpath{0,0}{10}{2,0,2,1,2,0,0,1,2,0}}
\caption{Motzkin path corresponding to $\pi = \protect\perm{2,1,6,3,4,7,5,8,10,9}$.}
\label{fig:motzkinP}
\end{figure}

\begin{thm}\label{thm:path2perm_map}
The map $\alpha:\Mo_n\to\MoP_n$ is bijective, and its restriction to $\Rio_n$ provides a bijection between $\Rio_n$ and $\RioP_n$.
\end{thm}

\def\lab#1{\color{gray}_{#1}}

\begin{proof}
Let $w=w_1\cdots w_n\in\MoP_n$, and let $u_1<\dots<u_k$ and $d_1<\dots<d_k$ be the indices of its $\Uu$-steps and $\Dd$-steps, respectively. Since $w$ never goes below the $x$-axis, the $j$-th $\Uu$ precedes the $j$-th $\Dd$, that is, $u_j<d_j$ for all $j$. We define $\beta(w)$ to be the permutation obtained by listing the indices $v$ with $w_v\not=\Dd$ in increasing order and inserting $d_j$ immediately before $u_j$ for each $j$. For instance, for the path 
\[ w = \Uu_{\lab{1}}\Dd_{\lab{2}}\Uu_{\lab{3}}\Hh_{\lab{4}}\Uu_{\lab{5}}\Dd_{\lab{6}}\Dd_{\lab{7}}\Hh_{\lab{8}}\Uu_{\lab{9}}\Dd_{\lab{10}} \]
shown in Figure~\ref{fig:motzkinP}, we list $1,3,4,5,8,9$ and insert $2,6,7,10$ before $1,3,5,9$, recovering the permutation $\perm{2,1,6,3,4,7,5,8,10,9}$. Note that the pairing of $d_j$ with $u_j$ differs in general from the usual matching of up and down steps.

Since $\pi=\beta(w)$ is a shuffle of two increasing sequences, it is $321$-avoiding. Every entry of $\pi$ preceding $d_j$ is either some $d_i$ with $i<j$ or an index of $w$ smaller than $u_j<d_j$, so $d_j$ is a LR maximum, and it is a descent since it is immediately followed by $u_j$. In particular, no $d_j$ occupies the last position of $\pi$.

Now let $v$ be an index with $w_v\not=\Dd$, let $j$ be the number of $\Uu$-steps before $w_v$, and let $y$ be the height of $w$ before step $v$. If $y>0$, then there are less than $j$ $\Dd$-steps to the left of $w_v$, so $d_j>v$. Since $u_j<v$ and both entries are listed in increasing order in $\pi$, the entry $d_j$ (inserted before $u_j$) lies to the left of $v$, and therefore $v$ is not a LR maximum.

If $y=0$, then $w_v=\Hh$, and there must also be $j$ $\Dd$-steps before $w_v$. Let $m\ge 0$ be the number of $\Hh$-steps before $w_v$. Thus $v-1=j+j+m$, so $v=2j+m+1$. Now, since the set $\{p<v: w_p\not=\Dd\}$ has $j+m$ elements, and the entries
$d_1,\dots,d_j$ are inserted to the left of $u_1,\dots,u_j$, respectively, there are precisely $2j+m$ entries to the left of $v$ in $\pi$. Thus $v$ is in position $2j+m+1$ and is therefore a fixed point of $\pi$.

Note that the last position of $\pi$ requires no separate treatment: the last step of $w$ is either a $\Dd$-step, in which case $\pi_n=u_k$ is not a LR maximum, or an $\Hh$-step, which necessarily lies on the $x$-axis, in which case the previous paragraph gives $\pi_n=n$. In conclusion, every LR maximum of $\pi$ is either a descent or a fixed point, hence $\pi=\beta(w)\in\Mo_n$.

The statement about the restriction to $\Rio_n$ follows if we show that
\begin{equation}\label{eq:fp_preserving} 
 \pi\in\Mo_n \text{ has a fixed point} \iff \alpha(\pi) \text{ has an $\Hh$-step on the $x$-axis.}
\end{equation}
If $\pi$ has a fixed point $\pi_j=j$, then $\pi=\sigma\oplus 1\oplus\tau$ with $\sigma\in\Mo_{j-1}$ and $\tau\in\Mo_{n-j}$. Therefore $\alpha(\pi)=\alpha(\sigma)\,\Hh\,\alpha(\tau)$, and since $\alpha(\sigma)$ is a Motzkin path, it ends at height $0$, so the displayed $\Hh$-step lies on the $x$-axis. Conversely, if $\alpha(\pi)=M'\,\Hh\,M''$ with the displayed $\Hh$-step on the $x$-axis, then $M'$ ends at height $0$ and $M''$ starts there, so both are Motzkin paths, say of lengths $m$ and $n-1-m$. Let $\sigma=\beta(M')\in\Mo_{m}$ and $\tau=\beta(M'')\in\Mo_{n-1-m}$. Then $\sigma\oplus 1\oplus\tau\in\Mo_n$ and $\alpha(\sigma\oplus 1\oplus\tau)=M'\,\Hh\,M''=\alpha(\pi)$. The injectivity of $\alpha$ then implies $\pi=\sigma\oplus 1\oplus\tau$, so $\pi$ has a fixed point.
\end{proof}

\begin{remark}
A bijection between Riordan paths and Riordan permutations was given by Menashe~\cite{Men07}: a Riordan path with up-steps at indices $u_1<\dots<u_k$ and down-steps at indices $d_1<\dots<d_k$ is sent to the permutation $d_1\,u_1\,H_1\,d_2\,u_2\,H_2\cdots d_k\,u_k\,H_k$, where $H_i$ is the increasing list of indices of the horizontal steps lying between $u_i$ and $u_{i+1}$. This is precisely the restriction of $\beta$ to $\RioP_n$.
\end{remark}

Since $M_n=R_n+R_{n+1}$ (Bernhart~\cite{Ber99}), the set of Motzkin paths of length $n$ with at least one $\Hh$-step on the $x$-axis is equinumerous with $\RioP_{n+1}$. Here is a concrete bijection. If $w\in\MoP_n\backslash\RioP_n$, then $w$ can be written uniquely as $w=w_0\,\Hh\,w_1$, where the displayed $\Hh$ is the first $\Hh$-step on the $x$-axis, so that $w_0$ is a Riordan path (possibly empty) and $w_1$ is a Motzkin path (possibly empty). Define $\gamma:\MoP_n\backslash\RioP_n \to \RioP_{n+1}$ by
\begin{equation*}
  \gamma(w_0\,\Hh\,w_1) = w_0\,\Uu\,w_1\,\Dd.
\end{equation*}
Since $w_0$ has no $\Hh$-step on the $x$-axis and $w_1$ is lifted to height at least $1$, the path $w_0\,\Uu\,w_1\,\Dd$ is indeed a Riordan path of length $n+1$. Conversely, every $v\in\RioP_{n+1}$ ends with a $\Dd$-step, and its matching $\Uu$-step starts on the $x$-axis. Writing $v=v_0\,\Uu\,v_1\,\Dd$ accordingly, the prefix $v_0$ ends at height $0$ and is therefore a Riordan path, while $v_1$ is a Motzkin path. Hence the inverse of $\gamma$ is given by $v\mapsto v_0\,\Hh\,v_1$, and $\gamma$ is a bijection.

\smallskip
Because of \eqref{eq:fp_preserving}, the map $\gamma$ transports to permutations.
\begin{prop}\label{prop:gamma_hat}
The map $\hat\gamma=\alpha^{-1}\circ\gamma\circ\alpha:\Mo_n\backslash\Rio_n\to\Rio_{n+1}$, where the inner $\alpha$ is taken on $\Mo_n$ and the outer $\alpha^{-1}$ on $\MoP_{n+1}$, is a bijection.
\end{prop}
\begin{proof}
By \eqref{eq:fp_preserving}, $\alpha$ restricts to bijections $\Mo_n\backslash\Rio_n\to\MoP_n\backslash\RioP_n$ and $\Rio_{n+1}\to\RioP_{n+1}$, so $\hat\gamma$ is a composition of three bijective maps.
\end{proof}

The map $\hat\gamma$ can be described explicitly as follows. For $\pi\in\Mo_n\backslash\Rio_n$, let $j$ be its first fixed point, let $m_1<\dots<m_k$ be the non-fixed LR maxima of $\pi$ lying to the right of position $j$, and set $m_{k+1}=n+1$. Then $\hat\gamma(\pi)$ is obtained from $\pi$ by replacing each $m_i$ by $m_{i+1}$ for $1\le i\le k$, and then inserting $m_1$ immediately before the fixed point $\pi_j=j$.

Table~\ref{tab:gamma_hat4} lists $\hat\gamma$ on the six elements of $\Mo_4\backslash\Rio_4$. In each case $\alpha(\pi)$ has its first $\Hh$-step on the $x$-axis at index $j$, and $\gamma$ replaces that step by $\Uu$ and appends a $\Dd$.

\begin{table}[ht]
\centering
\renewcommand{\arraystretch}{1.15}
\begin{tabular}{c >{\small$}c<{$} c >{\small$}c<{$} c c}
\toprule
$\pi$ & j & $\alpha(\pi)$ & m_1<\dots<m_{k+1} & $\gamma(\alpha(\pi))$ & $\hat\gamma(\pi)$ \\
\midrule
$\perm{1,2,3,4}$ & 1 & $\Hh\Hh\Hh\Hh$ & 5   & $\Uu\Hh\Hh\Hh\Dd$ & $\perm{5,1,2,3,4}$ \\
$\perm{1,2,4,3}$ & 1 & $\Hh\Hh\Uu\Dd$ & 4<5 & $\Uu\Hh\Uu\Dd\Dd$ & $\perm{4,1,2,5,3}$ \\
$\perm{1,3,2,4}$ & 1 & $\Hh\Uu\Dd\Hh$ & 3<5 & $\Uu\Uu\Dd\Hh\Dd$ & $\perm{3,1,5,2,4}$ \\
$\perm{1,4,2,3}$ & 1 & $\Hh\Uu\Hh\Dd$ & 4<5 & $\Uu\Uu\Hh\Dd\Dd$ & $\perm{4,1,5,2,3}$ \\
$\perm{2,1,3,4}$ & 3 & $\Uu\Dd\Hh\Hh$ & 5   & $\Uu\Dd\Uu\Hh\Dd$ & $\perm{2,1,5,3,4}$ \\
$\perm{3,1,2,4}$ & 4 & $\Uu\Hh\Dd\Hh$ & 5   & $\Uu\Hh\Dd\Uu\Dd$ & $\perm{3,1,2,5,4}$ \\
\bottomrule
\end{tabular}
\medskip
\caption{The bijection $\hat\gamma:\Mo_4\setminus\Rio_4\to\Rio_5$.}
\label{tab:gamma_hat4}
\end{table}

\section{Bijections between permutations and SYT}
\label{sec:perm2SYT_bijections}

The purpose of this section is to define a map $\Phi$ from $\Rio_n$ to $\RSYT(n)$ that extends to a bijection from $\Mo_n$ to $\SYT_{\le3}(n)$. Our map requires the construction of an auxiliary SYT with at most three rows that is then adjusted to meet the equal parity requirement.

Let $\pi\in\Rio_n$ have LR maxima $m_1<\cdots<m_k=n$ at positions $1=i_1<\cdots<i_k$. Since every $m_j$ is a descent, there has to be at least one entry to the right of it.

For $1\le j\le k-1$, we set
\[ a_j = \text{largest entry of $\pi$ between $m_j$ and $m_{j+1}$}, \]
and let $a_k$ be the last entry of $\pi$. Since $\pi$ is $321$-avoiding, the entries of $\pi$ that are not LR maxima increase from left to right. Thus $a_j=\pi_{i_{j+1}-1}$, $a_k=\pi_n$, and $a_1<a_2<\cdots<a_k$. Moreover, $a_j<m_j$ for every $j$, since $a_j$ is not a LR maximum and lies to the right of $m_j$. 

Let
\[ A(\pi) = \{a_1,\dots,a_k\} = \{\pi_{i-1}: \pi_i\in \LR(\pi),\, i>1\}\cup \{\pi_n\}, \]
let $r=n-2k$, and let $b_1<\cdots<b_r$ be the elements of $\{1,\dots,n\}\setminus\left(A(\pi)\cup\LR(\pi)\right)$.

Let $\theta(\pi)$ be the permutation defined by
\begin{equation}\label{eq:theta_pi}
 \theta(\pi)=m_1\,a_1\,m_2\,a_2\cdots m_k\,a_k\,b_1 \cdots b_r. 
\end{equation}
Let $P(\theta(\pi))$ be the insertion tableau obtained through the Robinson--Schensted algorithm applied to $\theta(\pi)$. Since $\theta(\pi)$ is a shuffle of three increasing sequences, it avoids the pattern $4321$, and so $P(\theta(\pi))$ has at most three rows by Schensted's theorem \cite{Sch61}.

\begin{prop}\label{prop:aux}
Let $\pi\in\Rio_n$ with $k=\lrm(\pi)$. Then $P(\theta(\pi))$ is a $\SYT$ of shape $(k+s,\,k,\,t)$ for some $s,t\ge0$ with $s+t=n-2k$ and $t\le k$. Moreover,
\begin{enumerate}[label=\textup{(\roman*)},itemsep=3pt]
  \item row $3$ of $P(\theta(\pi))$ consists of $t$ elements of $\LR(\pi)$;
  \item row $2$ has the remaining $k-t$ elements of $\LR(\pi)$ together with $t$ elements of $A(\pi)$;
  \item row $1$ has the remaining $k-t$ elements of $A(\pi)$ together with $b_1, \dots, b_r$;
\end{enumerate}
The map $\pi\mapsto P(\theta(\pi))$ is injective.
\end{prop}
\begin{proof}
Inserting $m_1\,a_1\cdots m_k\,a_k$ into the empty tableau gives the array with rows $a_1,\dots,a_k$ and $m_1,\dots,m_k$: each $m_j$ is appended to row $1$, and then $a_j$, with $a_{j-1}<a_j<m_j$, bumps $m_j$ to the end of row 2. For the remaining entries $b_1<\dots<b_r$, each insertion creates one new box, and by the row bumping lemma (Knuth \cite[Thm.~1]{Knu70} or Fulton \cite[\S1.1]{Fulton97}), the box created by $b_{i+1}$ lies strictly to the right of the one created by $b_i$. Since the new boxes lie in distinct columns and rows 1 and 2 already occupy columns $1,\dots,k$, no new box can lie in row 2, so the new boxes appear in row 1 beyond column $k$ or in row 3 within the first $k$ columns, in left-to-right order. This gives the shape $(k+s,k,t)$ with $s+t=r$, $t\le k$, and shows that the $t$ boxes of row 3 are created before the $s$ boxes of row 1.

A $b_i$ placed in row $1$ is never bumped because later insertions are larger. Hence only elements of $A(\pi)$ leave row 1, and they do so in increasing order: if $a$ is bumped by $b$ and later $a'$ by $b'>b$, then $a'>b'>b$ was in row 1 when $a$ was chosen as the smallest entry exceeding $b$, so $a<a'$. Thus an element of $A(\pi)$ entering row 2 exceeds all elements of $A(\pi)$ already there and bumps an element of $\LR(\pi)$ to row 3. This proves (i)--(iii).

For the injectivity, note that the RS map $\theta(\pi)\mapsto (P(\theta(\pi)),Q(\theta(\pi)))$, where $Q(\theta(\pi))$ is the corresponding recording tableau, is a bijective map, and $Q(\theta(\pi))$ depends only on $(k,s,t)$: inserting $m_1a_1\cdots m_ka_k$ places $1,3,\dots,2k-1$ in row 1 and $2,4,\dots,2k$ in row 2. By the left-to-right order established above, the next $t$ insertions determine the third row, so $2k+1,\dots,2k+t$ go into row 3 of $Q(\theta(\pi))$, and $2k+t+1,\dots,n$ go into row $1$. In other words, the shape of $P(\theta(\pi))$ determines $Q(\theta(\pi))$, and therefore $\theta(\pi)$. 

Finally $\theta(\pi)$ determines $\pi$: the $m_j$ are its LR maxima, and the entries of $\pi$ strictly between $m_j$ and $m_{j+1}$ are exactly those $v\notin\{m_1,\dots,m_k\}$ with $a_{j-1}<v \le a_j$ (with $a_0=0$), listed in increasing order.
\end{proof}

\subsection*{Parity correction}
On the set $\SYT_{\le3}(n)$ of standard Young tableaux of shape $(\lambda_1,\lambda_2,\lambda_3)$ with $\lambda_1 \ge \lambda_2 \ge \lambda_3 \ge 0$ and $\lambda_1+\lambda_2+\lambda_3 = n$, we consider two moves:
\medskip
\begin{enumerate}[itemsep=3pt]
\item[(M1)] (applied when $\lambda_3>0$ and $\lambda_3\not\equiv n\!\pmod{2}$) Let $y$ be the last entry of row 3 and let $x$ be the largest entry of row 2 with $x<y$. Remove $y$ from row 3, replace $x$ by $y$ in row 2, and append $x$ to the end of row 1. The shape becomes $(\lambda_1+1,\lambda_2,\lambda_3-1)$.
\item[(M2)] (applied when $\lambda_2>0$ and $\lambda_2\not\equiv n\pmod{2}$) Move the last entry of row 2 to the end of row 1. The shape becomes $(\lambda_1+1,\lambda_2-1,\lambda_3)$.
\end{enumerate}
\medskip
Every element of $\RSYT(n)$ is an element of $\SYT_{\le3}(n)$ with the additional property that 
\[ \lambda_1 \equiv \lambda_2 \equiv \lambda_3 \equiv n\! \pmod{2}. \]
A nonempty row whose length is not congruent to $n$ mod 2 is said to be \emph{defective}.

For $T$ in $\SYT_{\le3}(n)$, we let $\Psi(T)$ be the result of applying (M1) if row 3 is defective, and then applying (M2) if the resulting row 2 is defective. For $\pi\in \Rio_n$, we then define $\Phi$ by
\begin{equation}\label{eq:mapPhi}
 \Phi(\pi) = \begin{cases}
 	P(\theta(\pi)) &\text{if its rows have equal parity,} \\
	\Psi(P(\theta(\pi))) &\text{otherwise.}
	\end{cases}
\end{equation}

\begin{example}
For $\pi = \perm{4,1,2,8,3,9,5,6,7}\in\Rio_9$, we have $\theta(\pi)=\perm{4,2,8,3,9,7,1,5,6}$, so 
\[ P(\theta(\pi)) =\, \small \young(1356,279,48) \;\leadsto\;  \Phi(\pi) =\, \small \young(13567,289,4). \]
In this case, (M1) was the only parity correction needed. Here is an example of size 7 that requires both moves (M1) and (M2):
\[ \perm{3,1,2,7,4,5,6} \;\xrightarrow{\; \theta\; }\; \perm{3,2,7,6,1,4,5}\; \xrightarrow{\text{ RS }}\; \small \young(145,26,37) \; \xrightarrow{\text{(M1)}}\; \young(1456,27,3) \; \xrightarrow{\text{(M2)}}\; \young(14567,2,3). \]
\end{example}

\medskip
\begin{thm}\label{thm:main_bijection}
For every $n\in\mathbb{N}$, the map $\Phi:\Rio_n \to \RSYT(n)$ is a bijection.
\end{thm}

Before embarking on a proof, let us discuss why the map $\Phi$ is well-defined and let us work out some examples.

\begin{lem} \label{lem:legality}
For $\pi\in\Rio_n$, the moves \textup{(M1)} and \textup{(M2)} are legal, and $\Phi(\pi)\in\RSYT(n)$.
\end{lem}
\begin{proof}
For $\pi\in\Rio_n$ we let $\theta(\pi)$ be as in \eqref{eq:theta_pi} and let $T=P(\theta(\pi))$ be the corresponding insertion tableau. By Proposition~\ref{prop:aux}, $T$ has shape $(k+s,k,t)$ for some $s,t\ge0$ with $s+t=n-2k\equiv n\!\pmod{2}$ and $t\le k$. If $s>0$, some $b_i<\pi_n$ was appended to row~1, so $\pi_n=a_k$ must have been pushed to row~2; being the largest element of $A(\pi)$, it was the last to leave row~1, and no further bumping into row~2 occurs. Hence, for $s>0$, $\pi_n$ lies in row~2, the $\LR$ maximum $m$ it bumped is the last entry of row~3, and $t>0$. In all cases the last entry of row~1 is smaller than the last entry of row~2: for $s=0$ they share a column, and for $s>0$ the latter is at least $\pi_n$, which exceeds every entry of row~1.

If $T\in \RSYT(n)$, then $T=\Phi(\pi)$ and we are done. If $t>0$ and $t\not\equiv n\pmod{2}$, then $s$ must be odd. In this case, we perform (M1) on $T$ and let $V$ be the resulting tableau. This step moves $m$ from row 3 to row 2 at the current position of $\pi_n$ (reverting the forward RS insertion), and it moves $\pi_n$ from row 2 to the end of row 1. Since every entry in row 1 of $T$ is smaller than $\pi_n$, we have that $V$ is a legal $\SYT$ of shape $(k+s+1,k,t-1)$. As a consequence of (M1), the last entries of row 1 and row 3 of $V$ are both less than $m$.
 
Finally, suppose $k\not\equiv n\!\pmod 2$ (in particular $k>0$), and let $W$ denote $V$ if \textup{(M1)} was applied and $T$ otherwise. In either case, row~2 of $W$ has length $k$ and row~3 has length $t'\equiv n\pmod 2$, where $t'=t-1$ or $t'=t$. Since $t\le k$ and $t'\not\equiv k\!\pmod 2$, we get $t'\le k-1$, so removing the last box of row~2 leaves a valid shape. Let $x$ be the last entry of row~2 of $W$. Then $x$ is larger than the last entry of row~1 of $W$: for $W=T$ this was observed at the beginning of the proof, and for $W=V$ it follows from $x\ge m>\pi_n$. Hence (M2) is legal, and it yields a standard Young tableau of shape $(k+s+2,k-1,t-1)$ or $(k+s+1,k-1,t)$, all of whose row lengths have the parity of $n$.

In conclusion, $\Phi(\pi)\in\RSYT(n)$. Moreover, since row~2 is increasing, the last entry of row~1 of $\Phi(\pi)$ exceeds the last
entry of row~2 if \textup{(M2)} was applied, and is smaller than it otherwise (for $T$ by the initial observation, for $V$ because
$\pi_n<m$). This is what allows step~$(i)$ of the inverse algorithm below to detect whether (M2) was applied.
\end{proof}

\subsection*{Proof of Theorem~\ref{thm:main_bijection}}
We proceed by explicitly providing the inverse map $\Phi^{-1}$. Let $T$ be a tableau in $\RSYT(n)$ of shape $(\lambda_1,\lambda_2,\lambda_3)$. Whenever $\lambda_i>0$, we let $e_i$ denote the last entry of row $i$. If $\lambda_i=0$, set $e_i=0$. The map $\Phi^{-1}$ is defined by the following algorithm:
\begin{enumerate}[label=$(\roman*)$]
\item If $e_1<e_2$, set $T_2=T$. If $e_1>e_2$, then reverse (M2): move $e_1$ to the end of row 2 and denote the resulting $\SYT$ by $T_2$. 

\item Let $(\lambda_1',\lambda_2',\lambda_3')$ be the shape of $T_2$, let $e_1'$ be the last entry of row 1 of $T_2$, and let $m=\min\{x \text{ in row 2 of } T_2 : x>e_1'\}$. If $\lambda_1'=\lambda_2'$ or $m<e_3$, set $T_{2,1}=T_2$. If $\lambda_1' > \lambda_2'$ and $m>e_3$, then reverse (M1): move $e_1'$ to row 2 of $T_2$ at the current position of $m$, move $m$ to the end of row 3, and denote the resulting tableau by $T_{2,1}$. In either case, $T_{2,1}$ is a $\SYT$ of shape $(k+s,k,t)$ with some $s,t\ge 0$.

\item Let $R_{2,1}$ be the (recording) tableau with entries $1,3,\dots,2k-1, 2k+t+1,\dots,n$ in row 1, entries $2,4,\dots,2k$ in row 2, and entries $2k+1,\dots,2k+t$ in row 3, and let $\pi_0$ be the permutation corresponding to the pair $(T_{2,1},R_{2,1})$ via Robinson-Schensted.
\item By construction, $\pi_0$ is of the form \eqref{eq:theta_pi}, so we let $\pi = \Phi^{-1}(T)$ be the unique Riordan permutation such that $\pi_0=\theta(\pi)$. \qed
\end{enumerate}

\begin{example}
Consider $T=1267\,|\,35\,|\,48 \in \RSYT(8)$. The first two steps of the above algorithm give
\[ \young(1267,35,48) \;\xrightarrow{\; (i)\; }\; \young(126,357,48) \;\xrightarrow{\; (ii)\; }\; \young(126,357,48)\,, \]
and by step $(iii)$,
\[ \left(\small \,\young(126,357,48)\,,\; \young(135,246,78)\,\right) \;\xrightarrow{\; \text{RS}^{-1}\; }\; \pi_0 = \perm{4,1,5,3,8,7,2,6}. \]
Finally, $\Phi^{-1}(T) = \perm{4,1,5,2,3,8,6,7}$.
\end{example}

\begin{example} 
All Riordan permutations of size 5 correspond to $\SYT$ of shape $(3,1,1)$, and five of them require a parity correction on their insertion tableau (see Table~\ref{tab:Phi5}).
\end{example}

\begin{table}[ht]
\centering
\begin{tabular}{w{c}{4em} w{c}{6em} w{l}{5em} l}
\toprule
$\pi$ & $\theta(\pi)$ & $P(\theta(\pi))$ & \ $\Phi(\pi)$ \\
\midrule \\[-10pt]
$\perm{5,1,2,3,4}$ & $\perm{5,4,1,2,3}$ & $\scriptsize\young(123,4,5)$ & $\scriptsize\young(123,4,5)$ \\[12pt]
$\perm{4,1,2,5,3}$ & $\perm{4,2,5,3,1}$ & $\scriptsize\young(13,25,4)$ & $\scriptsize\young(135,2,4)$ \\[12pt]
$\perm{3,1,5,2,4}$ & $\perm{3,1,5,4,2}$ & $\scriptsize\young(12,34,5)$ & $\scriptsize\young(124,3,5)$ \\[12pt]
$\perm{4,1,5,2,3}$ & $\perm{4,1,5,3,2}$ & $\scriptsize\young(12,35,4)$ & $\scriptsize\young(125,3,4)$ \\[12pt]
$\perm{2,1,5,3,4}$ & $\perm{2,1,5,4,3}$ & $\scriptsize\young(13,24,5)$ & $\scriptsize\young(134,2,5)$ \\[12pt]
$\perm{3,1,2,5,4}$ & $\perm{3,2,5,4,1}$ & $\scriptsize\young(14,25,3)$ & $\scriptsize\young(145,2,3)$ \\[10pt]
\bottomrule
\end{tabular}
\bigskip
\caption{The map $\Phi:\Rio_5\to \RSYT(5)$.}
\label{tab:Phi5}
\end{table}

We conclude this section with an extension of the map $\Phi:\Rio_n \to \RSYT(n)$ from \eqref{eq:mapPhi} to the set $\Mo_n$ of ``Motzkin'' permutations as defined at the beginning of Section~\ref{sec:path2perm_bijections}. Recall that $\Mo_n$ consists of $321$-avoiding permutations of $[n]$ for which each of their LR maxima is either a descent or a fixed point.

Observe that the map $\varrho:\RSYT(n+1)\to \SYT_{\le3}(n)\setminus \RSYT(n)$, defined by removing the entry $n+1$, is bijective. For the inverse, note that since $\sum \lambda_i=n$, the number of rows $\not\equiv n\!\pmod 2$ is even, so a tableau outside $\RSYT(n)$ has exactly one row $\equiv n\!\pmod 2$, and entry $n+1$ can be appended only there (the row above, if there is one, has opposite parity and is therefore strictly longer).

With that in mind, and using the map $\hat\gamma$ from Proposition~\ref{prop:gamma_hat}, we now define
\begin{equation}\label{eq:mapPhi_tilde}
 \widetilde\Phi(\pi) = \begin{cases}
 	\Phi(\pi) &\text{if }\pi\in\Rio_n, \\
	(\varrho\circ\Phi\circ\hat\gamma)(\pi) &\text{if } \pi\in\Mo_n\backslash\Rio_n.
	\end{cases}
\end{equation}

\begin{thm}\label{thm:perm3syt_map}
The map $\widetilde\Phi: \Mo_n \to \SYT_{\le3}(n)$ is bijective, and its restriction to $\Rio_n$ provides a bijection between $\Rio_n$ and $\RSYT(n)$.
\end{thm}
\begin{proof}
By Theorem~\ref{thm:main_bijection}, the map $\Phi:\Rio_n \to \RSYT(n)$ is a bijection. Moreover, the map 
\[ \Mo_n\backslash\Rio_n \;\xrightarrow{\; \hat\gamma\;}\; \Rio_{n+1} \;\xrightarrow{\;\Phi\;}\; \RSYT(n+1) 
\;\xrightarrow{\; \varrho\;}\; \SYT_{\le3}(n)\backslash \RSYT(n) \]
is a composition of bijective maps.
\end{proof}

We finish this section with an example that illustrates the algorithm defined by $\widetilde\Phi^{-1}$. Suppose $T$ is the $\SYT$ given by $T = 13589 \,|\, 247A \,|\, 6$, where $A=10$. Since $T\not\in\RSYT$, we need to revert the map $\varrho\circ\Phi\circ\hat\gamma$. Letting $B=11$, we get
\[ {\small\young(13589,247A,6)} \xrightarrow{\; \varrho^{-1}\;} {\small\young(13589,247AB,6)} \xrightarrow{\; \Phi^{-1}\;} \perm{2,1,6,3,4,7,5,A,8,B,9} \xrightarrow{\; \hat\gamma^{-1}\;} \perm{2,1,6,3,4,7,5,8,A,9}. \]
This permutation corresponds to the Motzkin path $\Uu\Dd\Uu\Hh\Uu\Dd\Dd\Hh\Uu\Dd$ shown in Figure~\ref{fig:motzkinP}.

\section{Summary and further bijections via Dyck paths}
\label{sec:dyck_paths}

Combining Theorem~\ref{thm:path2perm_map} and Theorem~\ref{thm:perm3syt_map}, we obtain a bijection from Motzkin paths to $\SYT$ of height at most 3 that restricts to the corresponding Riordan subsets.

\begin{table}[ht!]
\centering
\renewcommand{\arraystretch}{1.15}
\begin{tabular}{w{c}{4em} w{c}{8em} w{l}{3.5em}}
\toprule
$\alpha(\pi)$ & $\pi$ & \ $\widetilde\Phi(\pi)$ \\
\midrule
\tikz[scale=0.35,baseline=0]{\motzkinpath{0,0}{4}{1,1,1,1}} & $\perm{1,2,3,4}$ & $\scriptsize\young(123,4)$ \\[8pt]
\tikz[scale=0.35,baseline=0]{\motzkinpath{0,0}{4}{1,1,2,0}} & $\perm{1,2,4,3}$ & $\scriptsize\young(13,2,4)$ \\[12pt]
\tikz[scale=0.35,baseline=0]{\motzkinpath{0,0}{4}{1,2,0,1}} & $\perm{1,3,2,4}$ & $\scriptsize\young(124,3)$ \\[8pt]
\tikz[scale=0.35,baseline=0]{\motzkinpath{0,0}{4}{1,2,1,0}} & $\perm{1,4,2,3}$ & $\scriptsize\young(12,3,4)$ \\[12pt]
\tikz[scale=0.35,baseline=0]{\motzkinpath{0,0}{4}{2,0,1,1}} & $\perm{2,1,3,4}$ & $\scriptsize\young(134,2)$ \\[8pt]
\tikz[scale=0.35,baseline=0]{\motzkinpath{0,0}{4}{2,1,0,1}} & $\perm{3,1,2,4}$ & $\scriptsize\young(14,2,3)$ \\[20pt]
\tikz[scale=0.35,baseline=0]{\motzkinpath{0,0}{4}{2,0,2,0}} & $\perm{2,1,4,3}$ & $\scriptsize\young(13,24)$ \\[6pt]
\tikz[scale=0.35,baseline=5]{\motzkinpath{0,0}{4}{2,2,0,0}} & $\perm{3,1,4,2}$ & $\scriptsize\young(12,34)$ \\[8pt]
\tikz[scale=0.35,baseline=1]{\motzkinpath{0,0}{4}{2,1,1,0}} & $\perm{4,1,2,3}$ & $\scriptsize\young(1234)$ \\[1pt]
\bottomrule
\end{tabular}
\bigskip
\caption{The bijections $\MoP_4\leftrightarrow\Mo_4\leftrightarrow\SYT_{\le3}(4)$. The last three rows show the restriction $\RioP_4\leftrightarrow\Rio_4\leftrightarrow\RSYT(4)$.}
\label{tab:path_perm_syt}
\end{table}

\begin{thm} \label{thm:path_syt}
The map $\widetilde\Phi \circ \alpha^{-1}:\MoP_n \to \SYT_{\le3}(n)$ is bijective, and its restriction to $\RioP_n$ provides a bijection between $\RioP_n$ and $\RSYT(n)$.
\end{thm}

The nine Motzkin objects (Motzkin paths, restricted $321$-avoiding permutations, and $\SYT$ of height at most 3) for $n=4$ are all listed in Table~\ref{tab:path_perm_syt}. The last three rows of the table show the corresponding Riordan objects.

\medskip
The bijections presented in this paper can be naturally extended, via Dyck paths, to include other sets of permutations and $\SYT$ counted by the Riordan numbers. 

\begin{thm} \label{thm:riordan_families}
The following Riordan families are in bijection:
\begin{enumerate}[leftmargin=25pt, itemsep=3pt, label=$(\alph*)$]
\item Dyck paths of semilength $n$ with $k$ peaks and no singleton ascent;
\item Dyck paths of semilength $n$ with $k$ peaks and no singleton descent;
\item 321-avoiding permutations on $[n]$ with $k$ $\LR$ maxima, all of which are descents;
\item 312-avoiding permutations on $[n]$ with $k$ $\LR$ maxima, all of which are descents;
\item Riordan paths of length $n$ with $k$ $\Dd$-steps $($and thus $k$ $\Uu$-steps$)$;
\item tableaux $T\in\RSYT(n)$ with $\lambda_2(T)+\chi(T)=k$, where $\chi(T)=1$ if the last entry of row~1 of $T$ exceeds the last entry of row~2, and $\chi(T)=0$ otherwise\footnote{Equivalently, the second row has length $k$ if $k\equiv n\!\pmod 2$, and $k-1$ if $k\not\equiv n\!\pmod 2$.};
\item $\SYT$ of shape $(k,k,1^{n-2k})$;
\item derangements of size $n$ having genus $0$ and $k$ cycles;
\end{enumerate}
\end{thm}

\begin{remark}
The map $\Phi$ carries $\lrm$ to the statistic $\lambda_2+\chi$ on $\RSYT(n)$. Indeed, for $\pi\in\Rio_n$ with $k=\lrm(\pi)$, the auxiliary tableau $P(\theta(\pi))$ has second row of length $k$ by Proposition~\ref{prop:aux}, and the parity correction shortens it by one exactly when $k\not\equiv n\!\pmod 2$, that is, exactly when (M2) is applied. By the last paragraph of the proof of Lemma~\ref{lem:legality}, (M2) is applied if and only if the last entry of row~1 of $\Phi(\pi)$ exceeds the last entry of row~2. Hence $\lambda_2(\Phi(\pi))+\chi(\Phi(\pi))=k$. The second row length alone does not determine $k$: permutations with $k$ $\LR$ maxima ($k\equiv n$) and with $k+1$ $\LR$ maxima ($k+1\not\equiv n$) both map to tableaux with $\lambda_2=k$.
\end{remark}

\tikzstyle{block} = [rectangle, draw, fill=gray!10, text width=14em, text centered, rounded corners, minimum height=3em]
\tikzstyle{line} = [draw, -latex]

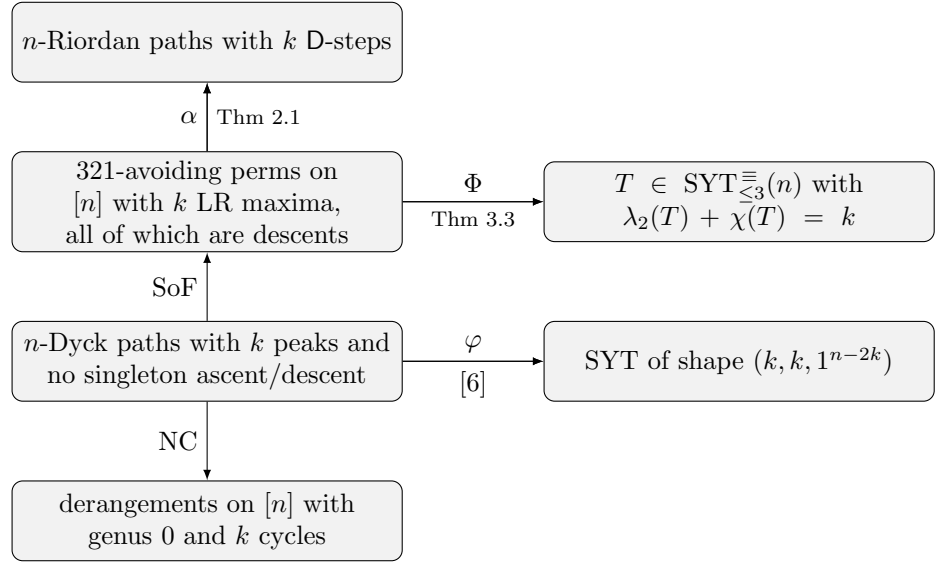
\begin{figure}[ht!]
\centering \small
\begin{tikzpicture}[node distance = 6em, auto]
  \node [block] (RPath) {$n$-Riordan paths with $k$ $\Dd$-steps};
  \node [block, below of=RPath] (RPerm) {321-avoiding perms on $[n]$ with $k$ $\LR$ maxima, all of which are descents};
  \node [block, below of=RPerm] (Dyck) {$n$-Dyck paths with $k$ peaks and no singleton ascent/descent};
  \node [block, below of=Dyck] (genus0) {derangements on $[n]$ with genus $0$ and $k$ cycles};
  \node [block, right of=RPerm,node distance=20em] (R-SYT) {$T\in\RSYT(n)$ with $\lambda_2(T)+\chi(T)=k$};
  \node [block, below of=R-SYT] (SYTFlag) {$\SYT$ of shape $(k,k,1^{n-2k})$};
  \path [line] (RPerm) -- node {$\Phi$}(R-SYT);
  \path [line] (RPerm) -- node[below] {\scriptsize Thm~\ref{thm:main_bijection}}(R-SYT);
  \path [line] (RPerm) --  node {$\alpha$} (RPath);
  \path [line] (RPerm) --  node[right] {\scriptsize Thm~\ref{thm:path2perm_map}} (RPath);
  \path [line] (Dyck) -- node {SoF}(RPerm); 
  \path [line] (Dyck) --  node[left] {NC} (genus0);
  \path [line] (Dyck) -- node {$\varphi$}(SYTFlag);
  \path [line] (Dyck) -- node[below] {\cite{GMTW20}}(SYTFlag);
  \end{tikzpicture}
  \caption{Bijections.}
  \label{fig:bijflow}
\end{figure}

The families listed in Theorem~\ref{thm:riordan_families} are bijectively connected through the maps described in Figure~\ref{fig:bijflow}. The bijection between (a) and (b) is the obvious one: swap $\Uu$'s and $\Dd$'s and read the resulting word backwards. For example, $\Uu\Uu\Dd\Uu\Uu\Dd\Dd\Dd$ corresponds to $\Uu\Uu\Uu\Dd\Dd\Uu\Dd\Dd$. The bijection between (c) and (d) is the well-known bijection $\psi:S_n(312)\to S_n(321)$ by Simion and Schmidt~\cite{SiSch85}, and the map between (b) and (c) is the one sometimes referred to as the sink-or-float bijection (see e.g.\ Vella~\cite{Vella03}), accredited to Krattenthaler~\cite{Kratt01} and E.~Deutsch. The bijection between (b) and (h) is essentially the known bijection between $n$-Dyck paths and noncrossing partitions of $[n]$ (see e.g.\ Callan~\cite{Callan08}). Finally, an explicit bijection $\varphi$ between (a) and (g) is given in \cite[Prop.~2.3]{GMTW20}.

\begin{figure}[ht!]
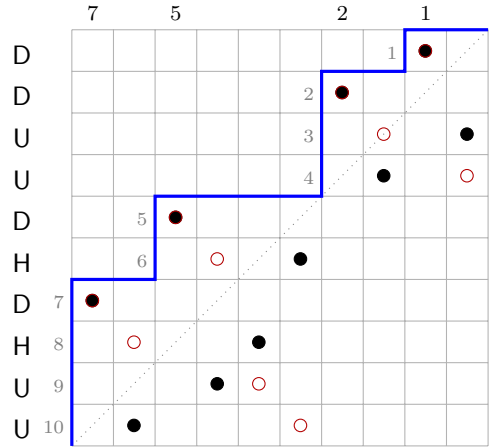

\tikz[scale=0.55]{
\plotpermutation{4,1,6,2,3,5,9,7,10,8}
\foreach [count=\i] \y in {4,3,6,5,2,1,9,8,10,7} {%
  \draw[right=14pt,above=14pt,red!70!black] (\i,\y) circle (\rad);
  }
\draw[gray,dotted] (1,1) -- (11,11);
\draw[very thick, blue] (1,1) -- (1,5) -- (3,5) -- (3,7) -- (7,7) -- (7,10) -- (9,10) -- (9,11) -- (11,11);
\foreach \y in {4,6,9,10} {\node[above=-1pt] at (-0.2,\y) {\small $\Dd$};}
\foreach \y in {1,2,7,8} {\node[above=-1pt] at (-0.2,\y) {\small $\Uu$};}
\foreach \y in {3,5} {\node[above=-1pt] at (-0.2,\y) {\small $\Hh$};}
\foreach [count=\i] \x in {9,7,7,7,3,3,1,1,1,1} {%
  \node[gray,above=7pt,left=-1pt] at (\x,11-\i) {\tiny \i};
  }
\foreach \x/\y in {1/9,2/7,5/3,7/1}{%
  \node[above=0pt] at (\y+0.5,11) {\scriptsize \x};
  }
}
\caption{Unifying example.}
\label{fig:uberExample}
\end{figure}

Most of these bijections can be nicely illustrated with a single picture. For example, for $n=10$ and $k=4$, consider the permutation $\pi = \perm{4,1,6,2,3,5,9,7,10,8}$, which belongs to the set in (c). Figure~\ref{fig:uberExample} shows the corresponding objects:

\begin{itemize}
\item Dyck path $P_1 = \Uu\Uu\Uu\Uu\Dd\Dd\Uu\Uu\Dd\Dd\Dd\Dd\Uu\Uu\Uu\Dd\Dd\Uu\Dd\Dd$ in (b).
\item Dyck path $P_2 = \Uu\Uu\Dd\Uu\Uu\Dd\Dd\Dd\Uu\Uu\Uu\Uu\Dd\Dd\Uu\Uu\Dd\Dd\Dd\Dd$ in (a).
\item 312-avoiding permutation $\perm{4,3,6,5,2,1,9,8,10,7}$ marked with red circles.
\item Riordan path $\Uu\Uu\Hh\Dd\Hh\Dd\Uu\Uu\Dd\Dd$.
\item Element of $\RSYT$ via our map $\Phi$: 
\[ \small \young(1238,457A,69)\,. \]
\item To obtain a $\SYT$ of shape $(4,4,1^2)$ from $P_2$, we follow the map $\varphi$ from \cite[Section~2]{GMTW20} and number its $\Dd$-steps from left to right, labeling the $\Uu$-steps at the peaks with the number of the matching $\Dd$-step:
\[ \Uu\;\Uu_{_{\color{blue}1}}\Dd_{\lab{1}}\Uu\;\Uu_{_{\color{blue}2}}\Dd_{\lab{2}}\Dd_{\lab{3}}\Dd_{\lab{4}}\Uu\;\Uu\;\Uu\;\Uu_{_{\color{blue}5}}\Dd_{\lab{5}}\Dd_{\lab{6}}\Uu\;\Uu_{_{\color{blue}7}}\Dd_{\lab{7}}\Dd_{\lab{8}}\Dd_{\lab{9}}\Dd_{\lab{A}}. \]
We then label the remaining $\Uu$-step from top to bottom on each ascent in a greedy fashion
\[ \Uu_{\lab{3}}\Uu_{_{\color{blue}1}}\Dd\;\Uu_{\lab{4}}\Uu_{_{\color{blue}2}}\Dd\;\Dd\;\Dd\;\Uu_{\lab{9}}\Uu_{\lab{8}}\Uu_{\lab{6}}\Uu_{_{\color{blue}5}}\Dd\;\Dd\;\Uu_{\lab{A}}\Uu_{_{\color{blue}7}}\Dd\;\Dd\;\Dd\;\Dd, \]
and read the partition $13 \,|\, 24 \,|\, 5689 \,|\, 7A$. This is then turned into the $\SYT$
\[ \small \young(1257,346A,8,9) \]
by trimming the blocks of the partition and moving the leftovers to the first column.
\item Finally, it is known (see Dulucq and Simion~\cite[Lemma~2.1]{DS98}) that a permutation on $[n]$ has genus 0 if and only if its cycle decomposition gives a noncrossing partition of $[n]$. From the Dyck path $P_2$ we obtain a noncrossing partition by labeling each $\Uu$-step with the number of its matching $\Dd$-step, and using the ascents as blocks: 
\[ 14 \,|\, 23 \,|\, 569A \,|\, 78. \]
The corresponding derangement of genus 0 is the permutation with cycle decomposition $(14)(23)(569A)(78)$, which in one-line notation becomes $\perm{4,3,2,1,6,9,8,7,10,5}$.
\end{itemize}
  

\end{document}